\documentclass{amsart}
\usepackage{amsmath,amssymb,amsthm}
\usepackage{bbm}
\usepackage[left=2cm, right=2cm, top=2cm, bottom=2cm]{geometry}
\theoremstyle{definition}
\newtheorem{defn}{Definition}[section]

\newtheorem{theorem}{Theorem}
\newtheorem*{theorem*}{Theorem}
\newtheorem{lem}[defn]{Lemma}
\newtheorem{coro}[defn]{Corollary}
\newtheorem{prop}[defn]{Proposition}
\newtheorem{rmk}[defn]{Remark}

\newcommand{\gap}{\, \vspace{1em}}

\newcommand{\F}{\mathbb{F}}
\newcommand{\M}{\mathcal{M}}
\newcommand{\I}{\mathcal{I}}

\let\oldproofname=\proofname
\renewcommand{\proofname}{\rm\bf{\oldproofname}}
\renewcommand{\leq}{\leqslant}
\renewcommand{\geq}{\geqslant}

\begin{document}

\pagenumbering{arabic}

\title{The function field Sath\'{e}-Selberg formula in arithmetic progressions and `short intervals'}

\author[A. Afshar]{Ardavan Afshar}
\address{Department of Mathematics\\University College London\\
25 Gordon Street, London, England}
\email{ardavan.afshar.15@ucl.ac.uk}

\author[S.Porritt]{Sam Porritt}
\address{Department of Mathematics\\University College London\\
25 Gordon Street, London, England}
\email{samuel.porritt.15@ucl.ac.uk}

\maketitle

\begin{abstract}
We use a function field analogue of a method of Selberg to derive an asymptotic formula for the number of (square-free) monic polynomials in $\F_q[X]$ of degree $n$ with precisely $k$ irreducible factors, in the limit as $n$ tends to infinity. We then adapt this method to count such polynomials in arithmetic progressions and short intervals, and by making use of Weil's `Riemann hypothesis' for curves over $\F_q$, obtain better ranges for these formulae than are currently known for their analogues in the number field setting. Finally, we briefly discuss the regime in which $q$ tends to infinity.
\end{abstract}

\section{Introduction}

One natural generalisation of the problem of counting primes up to $x$ is to count numbers up to $x$ with exactly $k$ distinct prime divisors. In \cite{Sathe}, Sath\'{e} proved that for $A >0 $ an arbitrary constant we have
$$\pi_k(x) := \#\{n \leq x : n = p_1\ldots p_k \text{ for some }  p_1, \ldots, p_k \text{ distinct primes} \} \sim G\left(\frac{k-1}{\log \log x}\right)\frac{x}{\log x} \frac{(\log \log x)^{k-1}}{(k-1)!} $$
uniformly for $x \geq 3$ and $1 \leq k \leq A \log \log x$, where $G(z) = \frac{1}{\Gamma(1+z)} \prod_{p \text{ prime}} (1+\frac{z}{p})(1-\frac{1}{p})^z$.\\
In \cite{Selberg}, Selberg gave a simpler proof of this result, now known as the ``Sath\'{e}-Selberg Formula". One might ask whether such a formula also holds for numbers restricted to a given arithmetic progression or short interval. For example, in \cite{Spiro}, Spiro showed that such a formula holds for $n \leq x$ restricted to $n \equiv a \mod q$, provided $q$ does not exceed some fixed power of $\log x$.

\gap

We begin by proving an asymptotic formula for the number of monic polynomials in $\F_q[X]$ of degree $n$ with exactly $k$ distinct irreducible divisors, using an adaptation of Selberg's technique. If we let $\M = \{f \in \F_q[X]  :  f \text{ monic} \}$, $\M_n = \{f \in \M  :  \deg f = n\}$ and $\I = \{p \in \M : p \text{ irreducible} \}$, then we get

\begin{theorem} \label{SS}
Let $A>1$. Then uniformly for all $n\geq 2$ and $1\leq k \leq A \log n$
\begin{align*}
\Pi_k(n) &:= \#\{f \in \M_n : f = p_1\ldots p_k \text{ for some }  p_1, \ldots, p_k \in \I \text{ distinct} \} \\
& = \frac{q^n}{n}\frac{(\log n)^{k-1}}{(k-1)!}\left(G\left( \frac{k-1}{\log n} \right)+O_A\left( \frac{k}{(\log n)^2} \right)\right)
\end{align*}
where $G(z)= \frac{F(1/q, z)}{\Gamma(1+z)}$ and $F(1/q, z) = \prod_{p \in \I}  \left(1+\frac{z}{q^{\deg p}}\right)\left(1-\frac{1}{q^{\deg p}}\right)^z$.
\end{theorem}
This says that the asymptotic density of polynomials in $\M_n$ with $k$ distinct prime divisors is $ \frac{1}{n}\frac{(\log n)^{k-1}}{(k-1)!} G\left( \frac{k-1}{\log n} \right) $. An asymptotic formula of this form was first derived by Car in \cite{Car}, but with an error term which inexplicitly depends on $k$ and $q$. 

\gap

With some additional technical work following Chapters II.5 and II.6 of \cite{Tenenbaum}, one could strengthen Theorem \ref{SS} to be of an analogous form to Chapter II.6 Theorem 4 of \cite{Tenenbaum}, namely
$$ \Pi_k(n)= \frac{q^n}{n} \left(\sum_{j=0}^J \frac{P_{j,k}(\log n)}{n^j} + O_{A}\left(\left(\frac{cJ+1}{n}\right)^{J+1}\frac{(\log n)^k}{k!}\right)\right) $$
where $P_{j,k}(x)$ is a polynomial of degree at most $k-1$, $J$ is a non-negative integer, and $c$ is some absolute constant. \\
Such an improvement could also be carried through to Theorems \ref{SS-AP} and \ref{SS-SI} below, to give similarly strengthened versions of what they state.
\gap

Next, we apply our method to Dirichlet L-functions for $\F_q[X]$, to derive an asymptotic formula for the number of such polynomials in a given arithmetic progression with difference of degree no bigger than roughly $n/2$.

\begin{theorem} \label{SS-AP}
Let $g,d \in \F_q[X]$ be coprime and $m= \deg d$. Let $A>1$, $n \geq 2$ and $ 1 \leq k \leq A\log n$. \\
Then for $m \leq \left(\frac{1}{2} - \frac{1+\log(1+\frac{A}{2})}{\log q}\right)n$ we have
\begin{align*} \Pi_k(n; g, d) &:= \#\{f \in \M_n \ f \equiv g \mod d : f = p_1\ldots p_k \text{ for some }  p_1, \ldots, p_k \in \I \text{ distinct} \}  \\
&= \frac{1}{\Phi(d)}\frac{q^{n}}{n}\frac{(\log n)^{k-1}}{(k-1)!}\left(G_d\left( \frac{k-1}{\log n}\right)+O_{A}\left( \frac{k}{(\log n)^2} \right)\right)
\end{align*}
where $\Phi(d) = \left|\left(\mathbb{F}_q[X]/(d(X))\right)^\times\right|$, and $G_d(z)= \left(\prod_{p|d} \left(1+\frac{z}{q^{\deg p}}\right)^{-1}\right) G(z)$ where $G(z)$ is defined as in Theorem \ref{SS}.
\end{theorem}
This range on the degree of the difference is obtained by our use of Weil's `Riemann Hypothesis', which allows us to bound the contributions from the non-principal characters as roughly square-root of the contribution from the principal character. A better range would require additional cancellation amongst these characters. \\ This range corresponds to taking the difference up to roughly $\sqrt{x}$ in the number field setting, compared to any fixed power of $\log x$ as in Theorem 1 of \cite{Spiro}.
\gap

Finally, by using an `involution-trick', we apply Theorem \ref{SS-AP} to derive an asymptotic formula for the number of such polynomials in a given `short interval' of length no shorter than roughly $n/2$ (which again corresponds to roughly $\sqrt{x}$ in the number field setting).

\begin{theorem} \label{SS-SI}
Let $g \in \F_q[X]$. Let $A>1$, $n \geq 2$ and $ 1 \leq k \leq A\log n$. \\
Then for $h$ satisfying $n - 1 \geq h \geq \left(\frac{1}{2} + \frac{1+\log(1+\frac{A}{2})}{\log q}\right)(n+1)$, we have
\begin{align*}
\Pi_k(n; g; h) &:= \#\{f \in \M_n \ \ \deg (f-g) \leq h : f = p_1\ldots p_k \text{ for some }  p_1, \ldots, p_k \in \I \text{ distinct} \} \\
&= \frac{q^{h+1}}{n}\frac{(\log n)^{k-1}}{(k-1)!}\left(H\left(\frac{k-1}{\log n}\right) + \frac{k-1}{q\log n}H\left(\frac{k-2}{\log (n-1)}\right)  + O_A\left( \frac{k}{(\log n)^2} \right) \right)
\end{align*}
where $H(z) = \frac{q}{q+z} G(z)$ and $G(z)$ is defined as in Theorem \ref{SS}.
\end{theorem}
The two main terms in Theorem \ref{SS-SI}	come from counting polynomials with non-zero constant term and polynomials with zero constant term separately. In the range where $k \asymp 
\log n$, the latter is roughly a factor of $q$ smaller than the former, and so of the same order of magnitude in the limit as $n$ tends to infinity. 
\gap

Note that, though Theorems \ref{SS}, \ref{SS-AP} and \ref{SS-SI} are relevant only in the regime in which $n$ tends to infinity, the error term in each does not depend on $q$. Moreover, we briefly discuss in Section \ref{qlimit} the regime in which $q$ tends to infinity.

\gap

\section{The function field Sath\'{e}-Selberg formula} \label{SS-Section}

\subsection{Outline}

Let $\omega(f) = \#\{p \in \I  :  p|f \}$  and define the M\"{o}bius function on $\M$ by 
$$\mu(f) = \begin{cases} (-1)^{\omega(f)}& \text{  if  } f \text{ is square-free}\\0& \text{   otherwise   }\end{cases} $$
so that $\mu^2$ is the indicator function for the square-free polynomials in $\M$ and
$$\Pi_k(n) = \sum_{\substack{f \in \M_n \\ \omega(f) = k }} \mu^2(f).$$ In order to study $\Pi_k(n)$, we will consider a two variable zeta function for $\M$ which will serve to count irreducible factors, namely,
$$ A(T, z) = \sum_{f \in \M} \mu^2(f) z^{\omega(f)}T^{\deg f} = \prod_{p \in \I} (1+zT^{\deg p}).$$ 
By taking $z \in \mathbb{C}$ and considering $A(T, z)$ as a power series in $T$ we will derive estimates for its coefficients, which we denote by $A_z(n)= \sum_{f \in \M_n} \mu^2(f) z^{\omega(f)}$. Then we can recover $\Pi_k(n)$ from the identity 
$$\sum_{k\geq 0}\Pi_k(n) z^k = A_z(n)$$ using Cauchy's formula $$\Pi_k(n) = \frac{1}{2\pi i}\oint \frac{A_z(n)}{z^{k+1}}dz.$$ 
This plan will be carried out by first deriving an estimate for the coefficients of the power series of $\zeta(T)^z$, where $\zeta(T) = \sum_{f \in \M} T^{\deg f}$ is the zeta function for $\M$, and then relating this to the estimate we want. Throughout, $A>1$ will be an arbitrary constant and $z$ a complex variable satisfying $|z|\leq A$.

\subsection{Proof of Theorem 1}

First note that there are $q^n$ polynomials in $\M_n$ and that therefore, for  $|T|<1/q$,
$$\zeta(T) = \sum_{f \in \M} T^{\deg f} = \sum_{n\geq 0}q^nT^n = \frac{1}{1-qT}.$$
For $T$ in this range, we define $\zeta(T)^z=\exp(z\log \zeta(T))$, where we choose the branch of the logarithm which is defined on the cut plane $\mathbb{C} \backslash [0,\infty)$ and is real for $T$ real. 

\begin{lem} \label{D_z}
If we define $D_z(n)$ for $n\geq 0$ via the identity $\zeta(T)^z = \sum_{n\geq 0}D_z(n)T^{n}$, then we have that $$D_z(n) = q^n \binom {n+z-1}{n} $$
where $ \binom{w}{n} = \frac{1}{n!}\prod_{j=0}^{n-1} (w-j) $.
\end{lem}
\begin{proof}
The binomial theorem gives us $$\zeta(T)^z = (1-qT)^{-z} = \sum_{n\geq 0}\binom {n+z-1} {n}q^nT^n. $$
\end{proof}
\begin{coro} \label{DzExp} For all $n\geq 1$ and $|z| \leq A$,
$$ D_z(n)= q^n\frac{n^{z-1}}{\Gamma(z)} + O_A\left(q^n n^{\Re{z}-2}\right). $$
\end{coro} 
\begin{proof} By choosing the implied constant large enough, it is sufficient to prove this for $n\geq 2A$. In this range, we consider two cases. The first is when $z$ is a non-positive integer, in which case $ D_z(n) = 0 = \frac{q^n}{\Gamma(z)} n^{z-1} $ . Otherwise we can use the Weierstrass Product Formula for $\Gamma(z)$ in the second line below to get
\begin{align*}
\frac{\Gamma(n+z)}{\Gamma(n+1)} &= \frac{1}{n+z} \left(\prod_{k=1}^{n} \frac{k + z}{k}\right) z\Gamma(z) \\
& = \frac{1}{n+z} \left(\prod_{k=1}^{n} \frac{k + z}{k}\right) e^{-\gamma z} \left(\prod_{k=1}^{\infty} \frac{k}{k+z} e^{z/k}\right) \\
& = \frac{e^{-\gamma z}}{n+z} \left(\prod_{k=1}^{n} e^{z/k}\right) \left(\prod_{k=n+1}^{\infty} \frac{k}{k+z} e^{z/k}\right) \\
& = \frac{e^{-\gamma z}}{n+z} \exp\left(\sum_{k=1}^n \frac{z}{k}\right) \exp\left(\sum_{k=n+1}^{\infty} \left(\frac{z}{k} - \log\left(1 + \frac{z}{k}\right)\right)\right) \\
& = \frac{e^{-\gamma z}}{n+z} \exp\left(z\left(\log n + \gamma + O\left(\frac{1}{n}\right)\right)\right) \exp\left(\sum_{k=n+1}^{\infty} \sum_{m = 2}^\infty (-1)^m \frac{z^m}{mk^m}\right) \\
& = \frac{n^z}{n+z} \left(1 + O_A\left(\frac{1}{n}\right)\right) \exp\left(O_A\left(\frac{1}{n}\right)\right) = n^{z-1}\left(1+O_A\left(\frac{1}{n}\right)\right).
\end{align*}
From this and Lemma \ref{D_z} we can conclude that
\begin{align*}
D_z(n) = q^n \binom {n+z-1}{n} = q^n \frac{\Gamma(n+z)}{\Gamma(n+1)\Gamma(z)} = q^n\frac{n^{z-1}}{\Gamma(z)}\left(1+O_A\left(\frac{1}{n}\right)\right).
\end{align*}
\end{proof}

It was fairly straightforward to derive an asymptotic formula for $D_z(n)$. The following technical proposition will allow us to use this result to deduce asymptotic formulae for the coefficients of more general series provided their behaviour at $1/q$ is similar to the singularity of $\zeta(T)^z$ at $T=1/q$.

\begin{prop}  \label{ModD} Let $ C(T, z) = \sum_{n \geq 0} C_z(n) T^n $ and $ M(T, z) = \sum_{n \geq 0} M_z(n) T^n $ be power series with coefficients depending on $z$ satisfying $ C(T, z) = M(T,z)\zeta(T)^z $. Suppose also that, uniformly for $|z|\leq A$,
\[\sum_{a\geq 0}\frac{|M_z(a)|}{q^{a}}a^{2A+2} \ll_{A} 1. \label{ConvCond} \tag{$\star$} \] 
Then, uniformly for $|z|\leq A$ and $n\geq 1$, we have $$C_z(n) = q^n\frac{n^{z-1}}{\Gamma(z)}M(1/q,z) + O_A(q^n n^{\Re z-2}).$$
\end{prop}
\begin{proof}
Using our expression for $D_z(n)$ from Corollary \ref{DzExp} and that $D_z(0) = 1$, we get
\begin{align*}
C_z(n) &= \sum_{0\leq a \leq n}M_z(a)D_z(n-a) \\
&= q^n \left[ \sum_{0\leq a < n}\frac{M_z(a)}{q^a}\frac{(n-a)^{z-1}}{\Gamma(z)} + O_A\left(\sum_{0\leq a < n} \frac{|M_z(a)|}{q^a}(n-a)^{\Re{z}-2}\right) + \frac{M_z(n)}{q^n} \right].
\end{align*}
Here we split the first sum at $n/2$ and use the fact that
\begin{equation*}
(n-a)^{z-1} = \begin{cases}
n^{z-1}\left(1+O_A(a/n)\right), &\text{ if } 0\leq a \leq n/2 \\
O_A(n^{A-1}), &\text{ if } n/2 < a < n.
\end{cases}
\end{equation*}
Combining this with (\ref{ConvCond}) we get 
\begin{align*}
\sum_{0\leq a < n}\frac{M_z(a)}{q^a}\frac{(n-a)^{z-1}}{\Gamma(z)} &= \sum_{0\leq a \leq n/2}\frac{M_z(a)}{q^a}\frac{n^{z-1}}{\Gamma(z)}\left(1+O_A(a/n)\right) + O_A\left(\sum_{n/2 <a < n} \frac{|M_z(a)|}{q^a}n^{A-1}\right) \\
&= \sum_{0\leq a \leq n/2}\frac{M_z(a)}{q^a}\frac{n^{z-1}}{\Gamma(z)} + O_A\left(n^{\Re z-2}\sum_{0\leq a \leq n/2}\frac{|M_z(a)|a}{q^{a}} + n^{\Re z-2} \sum_{n/2 <a < n} \frac{|M_z(a)|a^{2A+1}}{q^a}\right) \\
&= \frac{n^{z-1}}{\Gamma(z)}M(1/q,z) + O_A\left( n^{\Re z- 1}\sum_{a > n/2}\frac{|M_z(a)|}{q^a} + n^{\Re z-2} \right) \\
&= \frac{n^{z-1}}{\Gamma(z)}M(1/q,z) + O_A\left( n^{\Re z- 2}\sum_{a > n/2}\frac{|M_z(a)|a}{q^a} + n^{\Re z-2} \right) \\
&= \frac{n^{z-1}}{\Gamma(z)}M(1/q,z) + O_A\left(n^{\Re z-2} \right).
\end{align*}
Where, in the final term of the second line, we use that $ n^{\Re z-2} a^{2A+1} \gg n^{-A-2} n^{2A+1} = n^{A-1} $ for $ n/2 < a < n $. \\
Similarly, for the second sum we get
\begin{align*}
\sum_{0\leq a < n}\frac{|M_z(a)|}{q^a} (n-a)^{\Re{z}-2} &= \sum_{0\leq a \leq n/2}\frac{|M_z(a)|}{q^a} n^{\Re{z}-2} \left(1+O_A(a/n)\right) + O_A\left(\sum_{n/2 <a < n} \frac{|M_z(a)|}{q^a}n^{A-2}\right) \\
&\ll_A n^{\Re{z}-2} \sum_{0\leq a \leq n/2}\frac{|M_z(a)|}{q^a} + n^{\Re z-3} \sum_{0\leq a \leq n/2}\frac{|M_z(a)|a}{q^{a}} + n^{\Re z- 3} \sum_{n/2 <a < n} \frac{|M_z(a)|a^{2A+1}}{q^a} \\
&\ll_A n^{\Re{z}-2}.
\end{align*}
Finally,  by (\ref{ConvCond}) we have that the last term is
$$ \frac{M_z(n)}{q^n} \ll n^{\Re{z-2}}\frac{|M_z(n)|n^{A+2}}{q^n} \ll n^{\Re{z-2}} \sum_{a\geq 0}\frac{|M_z(a)|}{q^{a}}a^{A+2} \ll_A n^{\Re{z-2}}.  $$
Putting everything together proves the proposition.
\end{proof}
\begin{rmk}
This follows the same ideas as Theorem 7.18 of \cite{MontVaughan}.
\end{rmk}

We will apply the previous proposition with the series $F(T,z) = \sum _{n \geq 0} B_z(n) T^n$ defined by
$$F(T,z) := A(T,z)\zeta(T)^{-z} = \prod_{p \in \I}  (1+zT^{\deg p})(1-T^{\deg p})^z.$$ 
First we check that the conditions of Proposition \ref{ModD} are satisfied.

\begin{prop} \label{B-Conv1} For $|z|\leq A$, $n \geq 2$ and $\sigma \geq \frac{1}{2}$ 
$$\sum_{0 \leq a \leq n}\frac{|B_z(a)|}{q^{\sigma a}} \leq \begin{cases} c_{A, \sigma} &\text{ if } \sigma > \frac{1}{2} \\   n^{c_A} &\text{ if } \sigma = \frac{1}{2}, \end{cases} $$
where $c_{A,\sigma}$ is a constant depending on $A$ and $\sigma$, and $c_A$ is a constant depending on $A$. \\ \\ 
Consequently, since $a^{2A+2} \leq q^{a/3}$ for $a$ sufficiently large, we have for $|z|\leq A$ that
$$\sum_{a\geq 0}\frac{|B_z(a)|}{q^{a}}a^{2A+2} \ll_{A} \sum_{a\geq 0}\frac{|B_z(a)|}{q^{2a/3}}\ll_{A} 1. $$
\end{prop}
\begin{proof}
If we let $b_z(f)$ be the multiplicative function defined on powers of monic irreducible polynomials $p$ by the power series identity
\begin{equation*}
1+ \sum_{k\geq 1}b_z(p^k)S^k =(1+zS)(1-S)^z
\end{equation*}
then $F(T,z) = \sum_{f \in \M}b_z(f)T^{\deg f}$ and so $ B_z(n) = \sum_{f \in M_n} b_z(f) $. From this definition, we see that $b_z(p)=0$ on irreducible $p$ and, by Cauchy's inequality after integrating over the complex circle $|S|=\frac{1}{\sqrt{3/2}}$, that $$|b_z(p^k)|\leq (3/2)^{k/2}M_A, \text{   for   } k\geq 2$$ where $M_A=\sup_{|z|\leq A, |S|\leq \frac{1}{\sqrt{3/2}}}|(1+zS)(1-S)^z|$ is some constant depending on A. \\Therefore, letting $\M_{\leq n} = \{f \in \M  :  \deg f \leq n\}$ and
$\I_{\leq n} = \{p \in \I  :  \deg p \leq n\}$, we have
\begin{align*}
\sum_{0\leq a \leq n}\frac{|B_z(a)|}{q^{\sigma a}} &\leq \sum_{f \in \M_{\leq n}}\frac{|b_z(f)|}{q^{\sigma \deg f}} \\
&\leq \prod_{p \in \I_{\leq n}}\left(1+\sum_{k \geq 1}\frac{|b_z(p^k)|}{q^{k \sigma \deg p}}\right) \\
&\leq \prod_{p \in \I_{\leq n}}\left(1+M_A\sum_{k\geq 2} \left(\frac{\sqrt{3/2}}{q^{\sigma \deg p}} \right)^k \right) \\
&=\prod_{p \in \I_{\leq n}}\left(1+\frac{3M_A/2}{q^{\sigma\deg p}(q^{\sigma \deg p}-\sqrt{3/2})}\right).
\end{align*}
Taking the logarithm and using the prime polynomial theorem we get
\begin{align*}
\sum_{p \in \I_{\leq n}}\log\left(1+ \frac{3M_A/2}{q^{\sigma\deg p}(q^{\sigma \deg p}-\sqrt{3/2})}\right) & \leq 6M_A \sum_{1 \leq d \leq n}\frac{q^{d(1-2 \sigma)}}{d} \leq \begin{cases} \frac{6M_A}{q^{2\sigma-1}-1} &\text{ if } \sigma > \frac{1}{2} \\  12 M_A \log n &\text{ if } \sigma = \frac{1}{2}. \end{cases}
\end{align*}
Exponentiating then gives the stated result.
\end{proof}
\begin{rmk} \label{holo}
Proposition \ref{B-Conv1} also proves that $F(1/q, z)$ is absolutely uniformly convergent for $|z|\leq A$ and so holomorphic in $z$ for $|z|\leq A$.
\end{rmk}
\begin{rmk}
This follows the same ideas as the beginning of Chapter II.6 of \cite{Tenenbaum}.
\end{rmk}

\begin{coro} \label{A_z}
Uniformly for $|z|\leq A$ and $n\geq 1$, we have $$A_z(n) = q^n\frac{n^{z-1}}{\Gamma(z)}F(1/q,z) + O_A(q^n n^{\Re z-2}).$$
\end{coro}
\begin{proof}
By Proposition \ref{B-Conv1}, this follows from Proposition \ref{ModD} with $C(T,z) = A(T,z)$ and $M(T,z) = F(T,z)$.
\end{proof}

We now turn to the proof of a generalisation of the main result in this section.

\begin{prop} \label{General-SS}
Let $A>1$, $M(z)$ be a holomorphic function for $|z|\leq A$, and $C_z(n)$ be an arithmetic function such that uniformly for $|z|\leq A$ and $n\geq 1$
$$C_z(n) = q^n\frac{n^{z-1}}{\Gamma(z)}M(z) + O_A(q^n n^{\Re z-3/2}).$$
Moreover, for $k \geq 1$ an integer, let $\alpha_k(n)$ be the arithmetic function defined by
$$ \alpha_k(n) = \frac{1}{2\pi i}\oint \frac{C_z(n)}{z^{k+1}}dz. $$
Then for $N(z)= \frac{M(z)}{\Gamma(1+z)}$, we have that uniformly for all $n\geq 2$ and $1\leq k \leq A \log n$
$$ \alpha_k(n) =\frac{q^n}{n}\frac{(\log n)^{k-1}}{(k-1)!}\left(N\left( \frac{k-1}{\log n} \right)+O_A\left( \frac{k}{(\log n)^2} \right)\right).$$
\end{prop}
\begin{proof}
When $k=1$ we integrate around the circle $ |z| = 1/4 $ to get
$$ \alpha_1(n)  = \frac{1}{2\pi i}\oint \frac{C_z(n)}{z^2}dz = \frac{q^n}{n} \left(\frac{1}{2\pi i}\oint \frac{N(z)n^z}{z} dz +O_{A}(n^{-1/4})\right) = \frac{q^n}{n} \left(N(0)+ O_{A}(n^{-1/4})\right).$$
Now assume $k>1$. We integrate the around the circle $|z| = r = \frac{k-1}{\log n} < A$ so that the contribution from the error term in $C_z(n)$ is
$$O_{A} \left(q^n \int_{|z| = r} \left|\frac{n^{\Re z - 3/2} dz}{z^{k+1}}\right|\right) \ll_{A} q^n n^{r-3/2}r^{-k} = \frac{q^n}{n^{3/2}}e^{k-1}\frac{(\log n)^k}{(k-1)^k} \ll_{A}  \frac{q^n}{n^{3/2}}\frac{(\log n)^k}{(k-1)!}$$
which is smaller than the error which we are aiming for in the theorem. \\
The contribution from the main term in $C_z(n)$ is
$$\frac{q^n}{n} \int_{|z|=r}\frac {N(z)n^z}{z^{k}}dz.$$
Integration by parts gives
$$\int_{|z|=r}\frac {n^z}{z^{k-1}}dz = \frac{k-1}{\log n}\int_{|z|=r}\frac{n^z}{z^k}dz = r\int_{|z|=r}\frac{n^z}{z^k}dz \implies \frac{1}{2\pi i}\int_{|z|=r}(z-r)\frac{n^z}{z^{k}}dz = 0 .$$ 
Using this fact to determine that the last term in the following line vanishes, we have
\begin{align*}
\frac{1}{2\pi i}\int_{|z|=r}N(z)\frac{n^z}{z^{k}}dz &= \frac{N(r)}{2\pi i}\int_{|z|=r}\frac{n^z}{z^{k}}dz  +  \frac{1}{2\pi i}\int_{|z|=r}\left(N(z)-N(r)-N'(r)(z-r)\right) \frac{n^z}{z^{k}}dz \\
&= \frac{N(r)}{2\pi i}\int_{|z|=r}\frac{n^z}{z^{k}}dz  +   O \left(\left|\int_{|z|=r}N''(r)(z-r)^2\ \frac{n^z}{z^{k}}dz \right|\right) \\ 
&=  N(r)\frac{(\log n)^{k-1}}{(k-1)!} + O_A\left( \int_{|z|=r} |z-r|^2 \left|\frac{n^z}{z^k}\right|  |dz|\right)
\end{align*}
where $N(z)$ is a composition of holomorphic functions and so holomorphic for $|z| \leq A$, so in the final line we can use that $N''(z)$ is uniformly bounded for $|z| \leq A$  by a constant depending on $A$.\\
We can estimate this last integral as follows 
\begin{align*}
\int_{|z|=r} |z-r|^2 \left|\frac{n^z}{z^k}\right|  |dz| & = \int_{0}^{2\pi}r^{3-k}|e^{i\theta}-1|^2e^{r\cos \theta \log n} d\theta \\
& = r^{3-k}\int_{-\pi}^{\pi}4\sin^2(\theta/2) e^{(k-1)\cos \theta}d\theta\\
&\leq r^{3-k}\int_{-\pi}^{\pi}\theta^2 e^{(k-1)(1-\theta^2/5)}d\theta \\
& \leq r^{3-k}e^{k-1}\int_{-\infty}^{\infty}\theta^2 e^{-(k-1)(\theta^2/5)}d\theta \\
& \ll r^{3-k}e^{k-1}(k-1)^{-3/2}.
\end{align*}
The error is therefore
$$ \ll_A \frac{q^n}{n}\frac{e^{k-1}}{(k-1)^{(k-3/2)}} (\log n)^{k-3} \ll_A \frac{q^n}{n}\frac{k(\log n)^{k-3}}{(k-1)!}$$ by Stirling's approximation again and the result follows.
\end{proof}
\begin{rmk}
This follows the same ideas as Theorem 7.19 of \cite{MontVaughan}.
\end{rmk}

Now, taking $M(z) = F(1/q, z)$, $N(z) = G(z)$, $C_z(n) = A_z(n)$ and $\alpha_k(n) = \Pi_k(n)$ in Proposition \ref{General-SS}, and using Remark \ref{holo} and Corollary \ref{A_z} to verify its hypotheses, we prove Theorem \ref{SS}.

\begin{rmk} \label{EK}
We can also estimate $\rho_k(n) :=\{f \in \M_n  :  \omega(f) = k \} $ by first proving an analogue of Proposition \ref{B-Conv1} for the power series
$$ \tilde{F}(T,z):= \zeta(T)^{-z} \sum_{f \in \M} z^{\omega(f)}T^{\deg f}=\prod_{p \in \I}\left(1+ \frac{zT^{\deg p}}{1-T^{\deg p}}\right)(1-T^{\deg p})^z$$
then applying Proposition \ref{ModD} with $M(T, z) = \tilde{F}(T, z)$ and $C(T, z) = \tilde{A}(T, z)$ where
$$ \tilde{A}(T,z) = \sum_{n \geq 0} \tilde{A}_z(n) T^n := \sum_{f \in \M} z^{\omega(f)}T^{\deg f} = \tilde{F}(T,z) \zeta(T)^{z} $$
and finally applying Proposition \ref{General-SS} with $M(z) = \tilde{F}(1/q, z)$, $N(z) = \tilde{G}(z) = \frac{\tilde{F}(1/q, z)}{\Gamma(1+z)}$, $C_z(n) = \tilde{A}_z(n)$ and $\alpha_k(n) = \rho_k(n)$, in order to obtain an analogue of Theorem \ref{SS}, namely
$$\rho_k(n) = \frac{q^n}{n}\frac{(\log n)^{k-1}}{(k-1)!}\left(\tilde{G}\left( \frac{k-1}{\log n} \right)+O_A\left( \frac{k}{(\log n)^2} \right)\right)$$
uniformly for all $n\geq 2$ and $1\leq k \leq A \log n$. \\ \\
Using this, and following Theorem 7.20 and Theorem 7.21 of \cite{MontVaughan}, we can prove the analogue of the Erd\H{o}s-Kac theorem for $\F_q[T]$, which tells us the mean, variance and limiting distribution of the function $\omega$.
\end{rmk}

\gap

\section{The Sath\'{e}-Selberg formula in arithmetic progressions}

We now follow the same strategy, but with Dirichlet $L$-functions, in order to count polynomials, with a prescribed number of irreducible
factors, in arithmetic progressions. In the next section, we will see how this can then be used to count such polynomials from a ``short interval". 

\gap

Let $d\in \M$ be some polynomial of degree $m\geq 1$. Consider the characters $\chi: \left(\mathbb{F}_q[X]/(d(X))\right)^\times \longrightarrow \mathbb{C}^\times$, with $\chi_0$ being the principal character, and let
$$L(T, \chi) = \sum_{f \in \M} \chi(f) T^{\deg f} =\prod_{p \in \I}(1-\chi(p)T^{\deg p})^{-1} $$ be the associated $L$-function. 
As for $\zeta(T)^z$, we define $L(T,\chi)^z = \exp(z\log L(T,\chi))$ for $|T|<1/q$ where we choose the branch of the logarithm which is real for $T$ real.
Our first task is to relate the coefficients of $\zeta(T)^z$ and $L(T,\chi)^z$. Consider the following identities which follow from the binomial theorem,
$$\zeta(T)^z=\prod_{p \in \I}(1-T^{\deg p})^{-z}=\prod_{p \in \I}\left(1+\sum_{k \geq 1}\binom{z+k-1}{k}T^{k\deg p}\right)$$ 
$$L(T,\chi)^z = \prod_{p \in \I}(1-\chi(p)T^{\deg p})^{-z}=\prod_{p \in \I}\left(1+\sum_{k\geq 1}\binom{z+k-1}{k}\chi(p^k)T^{k\deg p}\right).$$ We see that if $d_z(f)$ is the multiplicative function defined on irreducible powers $p^k$ as $d_z(p^k) = \binom{z+k-1}{k}$ then $\zeta(T)^z = \sum_{f \in \M}d_z(f)T^{\deg f}$ and $L(T,\chi)^z = \sum_{f \in \M}d_z(f)\chi(f)T^{\deg f}$. Hence, $ D_z(n, \chi) := \sum_{f \in \M_n} d_z(f) \chi(f) $ is the coefficient of $ T^n$ in $L(T, \chi)^z$.

\subsection{Generalised divisor sums twisted by non-principal characters}

\begin{prop} \label{D_z-NP}
For $\chi \neq \chi_0$, $ |z| \leq A $ and $ n \geq 1 $
$$|D_z(n,\chi)| \leq  q^{n/2} \binom{n+Am-(A+1)}{n} \leq q^{n/2} \binom{n+Am}{n}. $$
\end{prop}
\begin{proof}
From Proposition 4.3 of \cite{Rosen}, we know that for $\chi \neq \chi_0$ we have
$$ L(T, \chi) = \sum_{j=1}^{m-1} \left(\sum_{f \in \M_j} \chi(f) \right) T^j = \prod_{j=1}^{m-1} (1-\alpha_jT)$$
where $|\alpha_j|$ is $0, 1$ or $\sqrt{q}$ as a consequence of Weil's Theorem (the `Riemannn Hypothesis' for curves over $\mathbb{F}_q$). \\
\\
Now, from the binomial theorem we get
$$ L(T, \chi)^z = \prod_{j=1}^{m-1} (1-\alpha_jT)^z = \sum_{n\geq 0} \left(\sum_{r_1 + \ldots + r_{m-1} = n} \binom{z}{r_1} \ldots \binom{z}{r_{m-1}} \alpha_1^{r_1} \ldots \alpha_{m-1}^{r_{m-1}} \right)(-1)^nT^n.  $$
Using that  $|\alpha_j| \leq \sqrt{q}$ and $ |z| \leq A$ we get that
\begin{align*}
|D_z(n,\chi)| &= \left| \sum_{r_1 + \ldots + r_{m-1} = n} \binom{z}{r_1} \ldots \binom{z}{r_{m-1}} \alpha_1^{r_1} \ldots \alpha_{m-1}^{r_{m-1}} \right| \\
& \leq \sum_{r_1 + \ldots + r_{m-1} = n} \left|\binom{z}{r_1}\right| \ldots \left|\binom{z}{r_{n-1}}\right| \sqrt{q}^{r_1 + \ldots + r_{m-1}} \\
& \leq q^{n/2} \sum_{r_1 + \ldots + r_{m-1} =n} \binom{A + r_1 - 1}{r_1} \ldots \binom{A + r_{m-1} - 1}{r_{m-1}}.
\end{align*}
Now, we recognise the sum as the coefficient of $T^n$ in the expansion of $$ ((1-T)^{-A})^{m-1} = (1-T)^{-A(m-1)} $$
which is also $\binom{n + A(m-1)-1}{n} = \binom{n+Am-(A+1)}{n} $. Indeed, this shows that the power series expansion of $L(T, \chi)^z$ is majorised by that of $(1-\sqrt{q}T)^{-A(m-1)}$. Since $ m, n \geq 1  $ we get that
$$|D_z(n,\chi)| \leq  q^{n/2} \binom{n+Am-(A+1)}{n} \leq q^{n/2} \binom{n+Am}{n}.$$
\end{proof}

\subsection{Formulae for $\Pi_k(n, \chi)$}
We are now interested in $\Pi_k$ twisted by a character, which we define as
$$\Pi_k(n, \chi) := \sum_{\substack{ f \in \M_n \\ \omega(f) = k }} \mu^2(f)\chi(f) $$
which, by analogy to Section \ref{SS-Section}, we relate to the generating function 
$$A(T, z, \chi) := \sum_{f \in \M} \mu^2(f) z^{\omega(f)} \chi(f) T^{\deg f} = \prod_{p \in \I} (1+z\chi(p)T^{\deg p})$$ 
whose power series coefficients are
$$A_z(n, \chi) := \sum_{f \in \M_n} \mu^2(f) \chi(f)   z^{\omega(f)}$$
so that, similarly to before
$$\sum_{k\geq 0}\Pi_k(n, \chi) z^k = A_z(n, \chi)$$
and by Cauchy's Theorem
$$ \Pi_k(n, \chi) = \frac{1}{2\pi i}\oint \frac{A_z(n, \chi)}{z^{k+1}}dz.$$ \\

Moreover, recall that we had
$$ F(T,z)=\sum_{f \in \M}b_z(f) T^{\deg f}= \prod_{p \in \I} (1+zT^{\deg p})(1-T^{\deg p})^z = A(T,z)\zeta(T)^{-z} $$
so we naturally define $F(T,z, \chi)$ by
$$ F(T,z, \chi) := \sum_{f \in \M}b_z(f)\chi(f) T^{\deg f} = \prod_{p \in \I} (1+\chi(p)zT^{\deg p})(1-\chi(p)T^{\deg p})^z = A(T, z, \chi)  L(T, \chi)^{-z} $$ 
and let $B_z(n, \chi) := \sum_{f \in \M_n}b_z(f)\chi(f)$ so that $ A_z(m, \chi)=\sum_{a+b = m}B_z(a, \chi)D_z(b, \chi) $.

\subsubsection{Non-principal characters}

In this subsection, $\chi$ will be a non-principal character.

\begin{lem} \label{B-Conv} For $|z|\leq A$ and $n \geq 2$
$$\sum_{0 \leq a \leq n}\frac{|B_z(a, \chi)|}{q^{a/2}} \leq  n^{c_A}  $$
where $c_A$ is a constant depending on A.
\end{lem}
\begin{proof}
$$ \sum_{0\leq a \leq n}\frac{|B_z(a, \chi)|}{q^{a/2}} \leq \sum_{f \in \M_{\leq n}}\frac{|b_z(f)|}{q^{\deg f/2}} \leq  n^{c_A}  $$
by the proof of Proposition \ref{B-Conv1}.
\end{proof}

We can use this to get an estimate for $A_z(n, \chi)$ as follows:

\begin{prop} \label{A_z-NP}
For $A > 1$ and $n \geq 2$
$$A_z(n, \chi) \leq q^{n/2} \binom{n + Am}{n} n^{c_A}.$$
\end{prop}
\begin{proof} Using Proposition \ref{D_z-NP} and Lemma \ref{B-Conv} we get
\begin{align*} 
A_z(n;\chi) &= \sum_{0\leq a \leq n}B_z(a, \chi) D_z(n-a, \chi) \\
&\leq q^{n/2}\sum_{0\leq a \leq n}\frac{|B_z(a,\chi)|}{q^{a/2}}\binom{n-a + Am}{n-a} \\
&\leq q^{n/2} \binom{n + Am}{n} \sum_{0\leq a \leq n} \frac{|B_z(a,\chi)|}{q^{a/2}} \leq q^{n/2} \binom{n + Am}{n} n^{c_A}.
\end{align*}
\end{proof}

We can now use Cauchy's Theorem to bound $\Pi_k(m ;\chi)$.
\begin{prop} For $A > 1$  and $n \geq 2$
$$ \Pi_k(n ; \chi) \leq q^{n/2} \binom{n + Am}{n} n^{c_A}. $$
\end{prop}
\begin{proof}
Recall the identity $$ \Pi_k(n ; \chi) = \frac{1}{2\pi i}\oint \frac{A_z(n;\chi)}{z^{k+1}}dz$$ where we take the contour to be the circle of radius $r=1$ centred at 0. \\
Then Proposition \ref{A_z-NP} gives us that this is
 $$\leq q^{n/2} \binom{n + Am}{n} n^{c_A} \frac{1}{2\pi}\oint \frac{|dz|}{|z|^{k+1}} \leq q^{n/2} \binom{n + Am}{n}n^{c_A} .$$
\end{proof}

\subsubsection{The principal character}

\begin{defn} Define $F_d$, $B^d_z$ and $b^d_z$ via the following formal power series equalities
$$ F_d(T, z) = \sum_{n\geq0} B^d_z(n) T^n =  \sum_{f \in \M}b^d_z(f) T^{\deg f} =  \prod_{p \nmid d} (1+zT^{\deg p})(1-T^{\deg p})^{z}\prod_{p|d} (1-T^{\deg{p}})^{z}.$$
\end{defn}

\begin{lem} \label{B-Conv-0} For $|z|\leq A$ and  $\sigma \geq \frac{2}{3}$ 
$$\sum_{a \geq 0}\frac{|B^d_z(a)|}{q^{\sigma a}} \ll_{A} \prod_{p|d} (1-q^{-\sigma \deg p})^{-A}.$$
\end{lem}
\begin{proof}
By making a change of variable $ S = T^{\deg p} $, we see that the multiplicative coefficients $ b^d_z(f) $ are defined on prime powers $f=p^k$ by the formal power series identity 
$$1+ \sum_{k\geq 1}b^d_z(p^k)S^k = \begin{cases} (1-S)^{z} &\text{\qquad if } p|d  \\ (1+zS)(1-S)^z &\text{\qquad if } p \nmid d. \end{cases}$$ 
So if $ p|d $, we have that $|b^d_z(p^k)|=|\binom{z}{k}|\leq\binom{A+k-1}{k}$, and if $p \nmid d$ we have that $b^d_z(p^k)=b_z(p)$. Therefore, we get
\begin{align*}
\sum_{a \geq 0}\frac{|B^d_z(a)|}{q^{\sigma a}} &\leq \sum_{f \in \M}\frac{|b^d_z(f)|}{q^{\sigma \deg f}} \\
&\leq \prod_{p|d} \left(1+\sum_{k \geq 1}\frac{|b^d_z(p^k)|}{q^{k\sigma \deg p}}\right)  \prod_{p\nmid d}\left(1+\sum_{k \geq 1}\frac{|b^d_z(p^k)|}{q^{k\sigma \deg p}}\right) \\
&\leq \prod_{p|d} \left(\sum_{k \geq 0} \binom{A+k-1}{k} {q^{-k\sigma \deg p}}\right) \prod_{p \in \I}\left(1+\sum_{k \geq 1}\frac{|b_z(p^k)|}{q^{k\sigma \deg p}}\right) \\
&=\prod_{p|d} (1-q^{-\sigma \deg p})^{-A}\sum_{f \in \M}\frac{|b_z(f)|}{q^{\sigma \deg f}}.
\end{align*}
Now, by the proof of Proposition \ref{B-Conv1}, $ \sum_{f \in \M}\frac{|b_z(f)|}{q^{\sigma \deg f}} \ll_A 1$ for $\sigma \geq \frac{2}{3}$, which gives the result.
\end{proof}

\begin{lem} \label{PhiEst} For $d \in \mathbb{F}_q[X]$ of degree $m \geq 1$ and $1 \geq \sigma > \frac{1}{2}$, we have 
$$ \prod_{p|d} (1-q^{-\sigma \deg p})^{-1} \leq (2+2\log m)^{8(qm)^{1-\sigma}}. $$
\end{lem}
\begin{proof}
Arrange the primes $p_1,\ldots,p_r$ dividing $d$ and the primes $P_1,\ldots$ in $\M $, in order of degree (where you can order those of the same degree arbitrarily). Then we must have that $ \deg P_i \leq \deg p_i $. \\
Now, for some $ N \in \mathbb{N} $, we have that $ \sum_{P: \deg P \leq N-1} \deg P < m \leq \sum_{P: \deg P \leq N} \deg P $. This means that $d$ has at most $ \# \{P: \deg P \leq N\} $ prime factors, and so, by the observation in the paragraph above
$$ \prod_{p|d} (1-q^{-\sigma \deg p})^{-1}  \leq \prod_{P:\deg P \leq N} (1-q^{-\sigma \deg P})^{-1} . $$
Taking the logarithm of the right hand side, and using the fact that $ -\log(1-\frac{1}{x}) \leq \frac{1}{x-1} $ for $x>1$, combined with the prime polynomial theorem, we get
$$\sum_{P:\deg P \leq N} -\log(1-q^{-\sigma \deg p}) \leq \sum_{r \leq N} \frac{\Pi(r)}{q^{\sigma r}-1} \leq 4\sum_{r \leq N} \frac{\Pi(r)}{q^{\sigma r}} \leq 4\sum_{r \leq N} \frac{q^{(1-\sigma)r}}{r} \leq 8q^{(1-\sigma)N}(\log (1+N))$$
where $\Pi(n) = \Pi_1(n) = \#\{f \in \M_n : f \text{ is prime}\}$.
Our choice of $N$ tells us that $q^N \leq qm $ (so $ N \leq (1+2\log m)$), since we have from the prime polynomial theorem that
$$m > \sum_{P: \deg P \leq N-1} \deg P = \sum_{r \leq N-1} \Pi(r) r \geq \sum_{r | N-1} \Pi(r) r = q^{N-1}.$$
Putting everything together we get that
$$ \prod_{p|d} (1-q^{-\sigma \deg p})^{-1} \leq \exp(8q^{(1-\sigma)N}(\log (1+N))) \leq (2+2\log m)^{8(qm)^{1-\sigma}}. $$
\end{proof}

\begin{prop} \label{B-Conv-Cor} For $|z|\leq A$  we have that $$\sum_{a\geq 0}\frac{|B^d_z(a)|}{q^{a}}a^{2A+2} \ll_{A}  (1 + \log m)^{K_A} $$ where $K_A$ is a constant depending on A.
\end{prop}
\begin{proof} When $\log m < 10A+10$ it suffices to show that $\sum_{a\geq 0}\frac{|B^d_z(a)|}{q^{a}}a^{2A+2} \ll_{A} 1 $. \\
This is indeed true in this case, since $ m \ll_A 1 $, and so by Lemma \ref{B-Conv-0} we have that for $ \sigma \geq \frac{2}{3}$
$$\sum_{a \geq 0}\frac{|B^d_z(a)|}{q^{\sigma a}} \ll_{A,\sigma} \prod_{p|d} (1-q^{-\sigma \deg p})^{-A} \ll_{A,\sigma} (1-q^{-\sigma})^{-Am} \ll_{A,\sigma} 1$$
and consequently that $\sum_{a\geq 0}\frac{|B^d_z(a)|}{q^{a}}a^{2A+2} \ll_{A} 1 $. \\ \\
When $\log m \geq 10A+10$, let $\tau = \frac{2A+2}{\log m \log q} \leq \frac{1}{5\log 2} \leq \frac{1}{3} $ so that $ 1 - \tau \geq \frac{2}{3} $ and moreover
$$ a \geq (\log m)^2 \implies (2A+2) \frac{\log a}{a} \leq (2A+2) \frac{2\log \log m}{(\log m)^2} \leq \frac{2A+2}{\log m} = \tau\log q \implies a^{2A+2} \leq q^{\tau a}.$$
So overall we have that $ a^{2A+2} \leq (\log m)^{4A+4} q^{\tau a} $.
Using this fact and Lemmas \ref{B-Conv-0} and \ref{PhiEst} we get that
\begin{align*}
\sum_{a\geq 0}\frac{|B^d_z(a)|}{q^{a}}a^{2A+2}  &\leq (\log m)^{4A+4} \sum_{a\geq 0}\frac{|B^d_z(a)|}{q^{(1-\tau)a}} \\
&\ll_A  (\log m)^{4A+4} \prod_{p|d} (1-q^{-(1-\tau) \deg p})^{-A} \\
&\ll_A  (\log m)^{4A+4}(2(1+\log m))^{8(qm)^{\tau}}\\
&\ll_A (1+\log m)^{K_A}.
\end{align*} 
\end{proof}

\begin{prop} \label{General-B-Prop}
Uniformly for $|z|\leq A$ and $n\geq 1$, we have 
\begin{align*} A_z(n, \chi_0) &=  q^n\frac{n^{z-1}}{\Gamma(z)}F_d(1/q,z) + O_{A}(q^n n^{\Re z-2}(1+\log m)^{K_A})  \\
&= \left(\prod_{p|d} \left(1+\frac{z}{q^{\deg p}}\right)^{-1} \right) F(1/q,z) q^n\frac{n^{z-1}}{\Gamma(z)} + O_{A}(q^n n^{\Re z-2}(1+\log m)^{K_A}).
\end{align*}
\end{prop} 
\begin{proof}
The first equality follows from the proof of Proposition \ref{ModD} (carrying throughout an additional factor of $(1+\log m)^{K_A}$ in the error term) and Proposition \ref{B-Conv-Cor} after noting that
$$A(T, z, \chi_0) = \prod_{p \in \I} (1+z\chi(p)T^{\deg p}) = \zeta(T)^z  \prod_{p \nmid d} (1+zT^{\deg p})(1-T^{\deg p})^{z} \prod_{p|d} (1-T^{\deg{p}})^{z} =  \zeta(T)^z   F_d(T, z).$$
The second equality follows from the observation that
$$ F_d(T, z) = \prod_{p \in \I} (1+zT^{\deg p})(1-T^{\deg p})^{z} \prod_{p|d} (1+zT^{\deg p})^{-1} = F(T, z)  \prod_{p|d} (1+zT^{\deg p})^{-1}. $$
\end{proof}	

We now turn to the proof of the main result of this subsection,

\begin{prop} \label{General-B-Thm}
Let $A>1$, $ \sqrt{n} \geq (1+\log m)^{K_A} $ and $G_d(z)= \left(\prod_{p|d} \left(1+\frac{z}{q^{\deg p}}\right)^{-1}\right) \frac{F(1/q, z)}{\Gamma(1+z)}$. Then 
$$ \Pi_k(n, \chi_0) = \frac{q^n}{n}\frac{(\log n)^{k-1}}{(k-1)!}\left(G_d\left( \frac{k-1}{\log n}\right)+O_{A}\left( \frac{k}{(\log n)^2} \right)\right)$$ uniformly for all $n\geq 2$ and $1\leq k \leq A \log n$.
\end{prop}
\begin{proof} For $|z| \leq A$, by Proposition \ref{General-B-Prop} and our condition on $n$, 
$$A_z(n, \chi_0) = \left(\prod_{p|d} \left(1+\frac{z}{q^{\deg p}}\right)^{-1} \right) F(1/q,z) q^n n^{z-1} +  O_{A}(q^n n^{\Re z-3/2}).$$ 
Now, we use Proposition \ref{General-SS} with $M(z) = \left(\prod_{p|d} \left(1+\frac{z}{q^{\deg p}}\right)^{-1} \right) F(1/q,z)$ (which is holomorphic for $|z| \leq A$ by Remark \ref{holo}), $N(z) = G_d(z)$, $C_z(n) = A_z(n, \chi_0)$ and $\alpha_k(n) = \Pi_k(n, \chi_0)$ to deduce the result.

\end{proof}

\subsection{Proof of Theorem 2}

We are now ready to present the proof of Theorem \ref{SS-AP}.

\begin{proof}[Proof of Theorem \ref{SS-AP}]
We use the orthogonality of characters,
$$ \sum_{ \substack{f \in \mathcal{M}_n \\ f \equiv g \mod d }} 1 = \frac{1}{\Phi(d)} \sum_{f \in \M_n} \sum_{\chi} \bar{\chi}(g)\chi(f) $$
where the sum is over characters $\chi: \left(\mathbb{F}_q[X]/(d(X))\right)^\times \longrightarrow \mathbb{C}^\times$ and $\Phi(d) = \left|\left(\mathbb{F}_q[X]/(d(X))\right)^\times\right|$, to get that
\begin{align*} 
\Pi_k&(n; g, c) = \sum_{ \substack{f \in \mathcal{M}_n \\ f \equiv g \mod d \\ \omega(f)=k }} \mu^2(f) \\
&= \frac{1}{\Phi(d)} \sum_{\chi} \bar{\chi}(g)\Pi_k(n, \chi) \\
&= \frac{1}{\Phi(d)} \Pi_k(n, \chi_0) + O\left(\frac{1}{\Phi(d)} \sum_{\chi \neq \chi_0} q^{n/2} \binom{n + Am}{n} n^{c_A}\right) \\
&= \frac{1}{q^m\prod_{p|d}(1-\frac{1}{q^{\deg p}})}\frac{q^n}{n}\frac{(\log n)^{k-1}}{(k-1)!}\left(G_d\left( \frac{k-1}{\log n}\right)+O_{A}\left( \frac{k}{(\log n)^2} \right)\right) + O\left( q^{n/2} \binom{n + Am}{n} n^{c_A}\right) \\
&= \left(\prod_{p|d}\left(1-\frac{1}{q^{\deg p}}\right)^{-1}\right)\frac{q^{n-m}}{n}\frac{(\log n)^{k-1}}{(k-1)!}\left(G_d\left( \frac{k-1}{\log n}\right)+O_{A}\left( \frac{k}{(\log n)^2} \right)\right) + O\left( q^{n/2} \binom{n + Am }{n} n^{c_A}\right)
\end{align*}
where we use Proposition \ref{D_z-NP} in the third line and Proposition \ref{General-B-Thm} (which is applicable since the condition $\left(\frac{1}{2} - \frac{1+\log(1+\frac{A}{2})}{\log q}\right)n \geq m$ implies the condition $ \sqrt{n} \gg_A (1+\log m)^{K_A} $) in the fourth line. \\
First note that, using Stirling's inequalities $ \sqrt{2\pi} n^{n+1/2} e^{-n} \leq n! \leq e n^{n+1/2} e^{-n} $ we get that for $a, b \geq 1 $
$$ \binom{a+b}{a} = \frac{(a+b)!}{a!b!} \leq \frac{e (a+b)^{a+b+1/2} e^{-(a+b)}}{2\pi a^{a+1/2}b^{b+1/2}e^{-(a+b)}} \leq \frac{e}{2\pi}\left(\frac{1}{a} + \frac{1}{b}\right)^{1/2} \left(1+\frac{b}{a}\right)^a \left(1+\frac{a}{b}\right)^b \leq \left(1+\frac{b}{a}\right)^a \left(1+\frac{a}{b}\right)^b. $$
Using this and the condition $\left(\frac{1}{2} - \frac{1+\log(1+\frac{A}{2})}{\log q}\right)n \geq m$ we get
$$ q^{n/2} \binom{n + Am}{n} n^{c_A + 2} \leq q^{n/2} \binom{n + \frac{A}{2}n}{n} n^{c_A+2} \leq q^{n/2} \left(1 + \frac{A}{2}\right)^n \left(1 + \frac{2}{A}\right)^{\frac{A}{2}n}  n^{c_A + 2} \ll_A q^{n/2} \left(1 + \frac{A}{2}\right)^n e^n \leq q^{n-m}.$$
From this, we then get that
\begin{align*}
\Pi_k(n; g, c) &= \left(\prod_{p|d}\left(1-\frac{1}{q^{\deg p}}\right)^{-1}\right)\frac{q^{n-m}}{n}\frac{(\log n)^{k-1}}{(k-1)!}\left(G_d\left( \frac{k-1}{\log n}\right)+O_{A}\left( \frac{k}{(\log n)^2} \right)\right) \\
&= \frac{1}{\Phi(d)}\frac{q^{n}}{n}\frac{(\log n)^{k-1}}{(k-1)!}\left(G_d\left( \frac{k-1}{\log n}\right)+O_{A}\left( \frac{k}{(\log n)^2} \right)\right).
\end{align*}
\end{proof}
\begin{rmk} \label{SS-AP-Rmk}
It is convenient for our proof of Theorem \ref{SS-SI} to restate the result of Theorem \ref{SS-AP} as it appears in the end of the proof, that is (under the same conditions as Theorem \ref{SS-AP})
$$ \Pi_k(n; g, c) = \left(\prod_{p|d}\left(1-\frac{1}{q^{\deg p}}\right)^{-1}\right)\frac{q^{n-m}}{n}\frac{(\log n)^{k-1}}{(k-1)!}\left(G_d\left( \frac{k-1}{\log n}\right)+O_{A}\left( \frac{k}{(\log n)^2} \right)\right).$$
\end{rmk}

\gap

\section{The Sath\'{e}-Selberg formula in short intervals} \label{SI-Section}

\subsection{The Involution Trick}

As in \cite{KeatRud}, we define the \emph{involution} of a polynomial $f\in \mathbb{F}_q[X]$ to be the polynomial $$f^*(X):=X^{\deg f}f(1/X).$$
The idea that such an involution links arithmetic progressions and short intervals has been known for a long time (see for example \cite{Hayes}). The following lemma, for example, appears as Lemma 4.2 in \cite{KeatRud}.

\begin{lem}
For $f\in \mathbb{F}_q[X]$ not divisible by $X$, $\omega(f^*) = \omega(f)$ and $\mu(f^*) = \mu(f)$.
\end{lem}
\begin{proof}
First of all, we note that for $f,g\in \mathbb{F}_q[X]$
$$ (fg)^*(X) = X^{\deg fg}fg(1/X) = X^{\deg f}f(1/X)X^{\deg g}g(1/X) = f^*(X)g^*(X). $$
Moreover, if $f\in \mathbb{F}_q[X]$ is not divisible by $X$, then $\deg f^*(X) = \deg f(X)$ so
$$ (f^*)^*(X) = X^{\deg f^*}f^*(1/X) = X^{\deg f^*} X^{-\deg f} f(X) = f(X).$$
Together, these imply that if $f = p_1^{a_1}\ldots p_r^{a_r} \in \F_q[X]$ where $p_i$ are distinct irreducibles none of which are $X$, then $f^* = (p^*_1)^{a_1}\ldots (p^*_r)^{a_r}$ where $p^*_i$ are distinct irreducibles none of which are $X$. So, if $f\in \mathbb{F}_q[X]$ is not divisible by $X$, then $\omega(f^*) = \omega(f)$ and $\mu(f^*) = \mu(f)$.
\end{proof}

In order to apply our result concerning polynomials from an arithmetic progression to prove one about polynomials belonging to a short interval, we use the following observation.

\begin{lem} \label{Inv}
Let $f$ and $g$ be polynomials of degree $n$ and $h$ an integer $\leq n$. Then $\deg(f-g) \leq h$ if and only if $f^* \equiv g^* \mod X^{n-h}.$ 
\end{lem}
\begin{proof}
Write $$f(X) = a_n X^n + \ldots + a_h X^h + \ldots + a_0 $$
$$g(X) = b_n X^n + \ldots  + b_h X^h + \ldots + b_0 $$ where $a_n$ and $b_n$ are non-zero. Then $$f^*(X) = a_n + \ldots + a_h X^{n-h} + \ldots + a_0X^n $$
$$g^*(X) = b_n + \ldots  + b_h X^{n-h} + \ldots + b_0X^n.$$ From this we can see that each condition is satisfied if and only if $a_i=b_i$ for each $i= h+1, \ldots, n.$
\end{proof}

Notice that $f^*$ and $g^*$ have non-zero constant terms.

\subsection{Proof of Theorem 3}



We first split the sum defining $\Pi_k(n; g; h)$ into two
$$\Pi_k(n; g; h) =\sum_{ \substack{f \in \mathcal{M}_n \\ \deg(f-g) \leq h \\ \omega(f)=k }} \mu^2(f) =\sum_{ \substack{f \in \mathcal{M}_n \\ \deg(f-g) \leq h \\ \omega(f)=k \\ f(0)\neq0}} \mu^2(f) + \sum_{ \substack{f \in \mathcal{M}_n \\ \deg(f-g) \leq h \\ \omega(f)=k \\ f(0) = 0}} \mu^2(f).$$ 
Using Lemma \ref{Inv} on the first sum we get
\begin{align*}
\sum_{ \substack{f \in \mathcal{M}_n \\ f^* \equiv g^* \mod X^{n-h} \\ \omega(f^*)=k \\ \deg f^* = n}} \mu^2(f^*) &= \sum_{ \substack{ \deg f = n \\ f \equiv g^* \mod X^{n-h} \\ \omega(f)=k }} \mu^2(f) \\
&= \sum_{a \in \mathbb{F}_q^*} \sum_{ \substack{ f \in \mathcal{M}_{n} \\ f \equiv a^{-1}g^* \mod X^{n-h} \\ \omega(f)=k }}\mu^2(f) \\
&=\sum_{a \in \mathbb{F}_q^* } \Pi_{k}(n;a^{-1}g^*, X^{n-h}).
\end{align*}
Since $a^{-1}g^*$ has non-zero constant term for each $a \in \mathbb{F}_q^*$, and from the condition of the theorem we have that $\left(\frac{1}{2}-\frac{1+\log(1+\frac{A}{2})}{\log q}\right) n > n-h  \geq 1$, so we may apply Remark \ref{SS-AP-Rmk} to get
\begin{align*}
&(q-1)\frac{q}{q-1}\frac{q^{h}}{n}\frac{(\log n)^{k-1}}{(k-1)!}\left(H\left( \frac{k-1}{\log n}\right)+O_A\left( \frac{k}{(\log n)^2} \right)\right) \\
&=\frac{q^{h+1}}{n}\frac{(\log n)^{k-1}}{(k-1)!}\left(H\left( \frac{k-1}{\log n}\right)+O_A\left( \frac{k}{(\log n)^2} \right)\right).
\end{align*}
Now, we split the second sum into two sums, the latter of which is zero, and then apply Lemma \ref{Inv} to the former
\begin{align*}\sum_{ \substack{f \in \mathcal{M}_{n-1} \\ \deg(Xf-g)\leq h \\ \omega(Xf)=k }} \mu^2(Xf) &= \sum_{ \substack{f \in \mathcal{M}_{n-1} \\ \deg(Xf-g)\leq h \\ \omega(f)=k-1 \\ f(0)\neq 0}} \mu^2(Xf) + \sum_{ \substack{f \in \mathcal{M}_{n-1} \\ \deg(Xf-g) \leq h \\ \omega(f)=k \\f(0)=0}} \mu^2(Xf) \\
&= \sum_{ \substack{\deg f = n-1 \\ f \equiv g^* \mod X^{n-h} \\ \omega(f)=k-1  }}\mu^2(f) \\
&= \sum_{a \in \mathbb{F}_q^*} \sum_{ \substack{ f \in \mathcal{M}_{n-1} \\ f \equiv a^{-1}g^* \mod X^{n-h} \\ \omega(f)=k-1 }}\mu^2(f) \\
&=\sum_{a \in \mathbb{F}_q^* } \Pi_{k-1}(n-1;a^{-1}g^*, X^{n-h}).
\end{align*}
Since $a^{-1}g^*$ has non-zero constant term and for each $a \in \mathbb{F}_q^*$, and from the condition of the theorem we have $\left(\frac{1}{2}-\frac{1+\log(1+\frac{A}{2})}{\log q}\right)(n-1) \geq n-h  \geq 1 $, we may apply Remark \ref{SS-AP-Rmk} again to get
\begin{align*}
&(q-1)\frac{q}{q-1}\frac{q^{h-1}}{n-1}\frac{(\log (n-1))^{k-2}}{(k-2)!}\left(H\left( \frac{k-2}{\log (n-1)}\right)+O_A\left( \frac{k-1}{(\log (n-1))^2} \right)\right) \\
&=\frac{q^h}{n}\frac{(\log n)^{k-2}}{(k-2)!}\left(H\left( \frac{k-2}{\log (n-1)}\right)+O_A\left( \frac{k}{(\log n)^2} \right)\right) \\
&=\frac{q^{h+1}}{n}\frac{(\log n)^{k-1}}{(k-1)!}\left(\frac{k-1}{q\log n}H\left( \frac{k-2}{\log (n-1)}\right)+O_A\left( \frac{k}{(\log n)^2} \right)\right).  
\end{align*}
Putting everything together proves the Theorem.

\gap

\section{The $q$-limit} \label{qlimit}

We conclude by briefly discussing what happens in the regime in which $q$ tends to infinity. First, note that
\begin{align*}
\Pi_k(n) = \sum_{\substack{f \in \M_n \\ \omega(f) = k }} \mu^2(f) = \frac{1}{k!}\sum_{\substack{p_1, \ldots, p_k \in \I \\ \text{pairwise distinct} \\ \deg(p_1\ldots p_k) = n }} 1 
&= \frac{1}{k!}\left(\sum_{\substack{p_1, \ldots, p_k \in \I \ \\ \deg(p_1\ldots p_k) = n }} 1 + O\left(\binom{k}{2} \sum_{\substack{p_1, \ldots, p_k \in \I \\ p_{k-1} = p_k \\ \deg(p_1\ldots p_{k}) = n}} 1 \right)\right) \\
&= \frac{1}{k!}\left(\sum_{\substack{p_1, \ldots, p_k \in \I \ \\ \deg(p_1\ldots p_k) = n }} 1 + O\left(k^2 \sum_{\substack{p_1, \ldots, p_{k-1} \in \I \ \\ \deg(p_1\ldots p_{k-1)} \leq n-1 }} 1 \right)\right)
\end{align*}
where the error term comes from bounding the over count by terms where (at least) two of the $p_i$ are the same. Now, using the prime polynomial theorem, and taking $k = O(q)$ for the third equality below we get that our sum is
\begin{align*}
\sum_{\substack{p_1, \ldots, p_k \in \I \ \\ \deg(p_1\ldots p_k) = n }} 1  &= \sum_{\substack{n_1 + \ldots + n_k = n \\ n_i \geq 1}} \prod_{i=1}^k \frac{1}{n_i} (q^{n_i} + O(q^{\lfloor n_i/2 \rfloor})) \\
&= q^n \sum_{\substack{n_1 + \ldots + n_k = n \\ n_i \geq 1}} \frac{1}{n_1 \ldots n_k}  (1 + O(1/q))^k \\
&= \frac{q^n}{n} \sum_{\substack{n_1 + \ldots + n_k = n \\ n_i \geq 1}} \frac{n_1 + \ldots + n_k}{n_1 \ldots n_k} (1 + O(k/q))\\
&= \frac{q^n}{n} k \sum_{\substack{n_1 + \ldots + n_{k-1} \leq n-1 \\ n_i \geq 1}} \frac{1}{n_1 \ldots n_{k-1}} (1 + O(k/q)).
\end{align*}
Similarly, using the second equality above, and again taking $k = O(q)$, the sum in the error term is
\begin{align*}
\sum_{\substack{p_1, \ldots, p_{k-1} \in \I \ \\ \deg(p_1\ldots p_{k-1)} \leq n-1 }} 1 &= \sum_{r \leq n-1} \sum_{\substack{p_1, \ldots, p_{k-1} \in \I \ \\ \deg(p_1\ldots p_{k-1)} = r}} 1 \\
&\leq \sum_{r \leq n-1} q^r \sum_{\substack{n_1 + \ldots + n_{k-1} = r \\ n_i \geq 1}} \frac{1}{n_1 \ldots n_{k-1}}  (1 + O(1/q))^k \\
&\leq q^{n-1} \sum_{\substack{n_1 + \ldots + n_{k-1} \leq n-1 \\ n_i \geq 1}} \frac{1}{n_1 \ldots n_{k-1}}(1 + O(k/q)). 
\end{align*}
Putting these results together we get, as long as $k = O(q)$, that
\begin{align*}
\Pi_k(n) = \frac{q^n}{n} \frac{1}{(k-1)!} \sum_{\substack{n_1 + \ldots + n_{k-1} \leq n-1 \\ n_i \geq 1}} \frac{1}{n_1 \ldots n_{k-1}}(1 + O(kn/q))
\end{align*}
which gives us an asymptotic formula for $\Pi_k(n)$ as $q \to \infty$, as long as $k = o(q/n)$. \\ \\
Moreover, note that, when $k = O(\log n/\log\log n)$, we have that
$$ \log^{k-1} n \left(1 + O\left(\frac{k\log k}{\log n}\right)\right) = \left(\sum_{r \leq \frac{n-1}{k}} \frac{1}{r} \right)^{k-1} \leq \sum_{\substack{n_1 + \ldots + n_{k-1} \leq n-1 \\ n_i \geq 1}} \frac{1}{n_1 \ldots n_{k-1}} \leq \left(\sum_{r \leq n} \frac{1}{r} \right)^{k-1} = \log^{k-1} n \left(1 + O\left(\frac{k}{\log n}\right)\right) $$
so that, as long as $k = o(\log n/\log\log n)$, we get, as $n \to \infty$, that
$$ \frac{q^n}{n} \frac{1}{(k-1)!} \sum_{\substack{n_1 + \ldots + n_{k-1} \leq n-1 \\ n_i \geq 1}} \frac{1}{n_1 \ldots n_{k-1}} \sim \frac{1}{n} \frac{\log^{k-1} n}{(k-1)!}.$$
This agrees with Theorem \ref{SS} in this range for $k$ (after noting that $G(z) = (1 + o(|z|))$). \\ \\
One can use similar elementary calculations, along with the input of the prime polynomial theorem in arithmetic progressions, to get asymptotic formulae for $\Pi_k(n; g, d)$, and consequently $\Pi_k(n; g; h)$ using the involution trick (as in Section \ref{SI-Section}) for the same range of $k$.

\gap

\section*{Acknowledgements}

The authors would like to thank Andrew Granville for his encouragement, thoughtful advice and engaging discussions. This work was supported by the Engineering and Physical Sciences Research Council EP/L015234/1 via the EPSRC Centre for Doctoral Training in Geometry and Number Theory (The London School of Geometry and Number Theory), University College London. We would also like to thank Vlad Matei and the anonymous referee for pointing out corrections and missing references.

\end{document}